\documentclass[a4paper,12pt,intlimits,oneside]{amsart}

\usepackage{enumerate}
\usepackage{amsfonts,amsmath}

\usepackage{latexsym,amssymb}

\usepackage{mathrsfs}

\newcommand{\comment}[1]{}

\newcommand{\bR}{{\mathbb R}}

\newcounter{rek}
\title[the product space $H^1_L\times BMO_L$]{Bilinear decompositions for the product space $H^1_L\times BMO_L$}         

\author{Luong Dang Ky}


\date{}
\begin{document}
 \maketitle

\begin{abstract}
In this paper, we improve a recent result by Li and Peng on products  of functions in $H_L^1(\bR^d)$ and $BMO_L(\bR^d)$, where $L=-\Delta+V$ is a Schr\"odinger operator with $V$ satisfying an appropriate reverse H\"older inequality. More precisely, we prove that such products may be written as the sum of two continuous bilinear operators, one
from $H_L^1(\bR^d)\times BMO_L(\bR^d) $ into $L^1(\bR^d)$, the other one from $H^1_L(\bR^d)\times BMO_L(\bR^d) $ into  $H^{\log}(\bR^d)$, where the space $H^{\log}(\bR^d)$ is the set of distributions $f$ whose grand maximal function $\mathfrak Mf$ satisfies
$$\int_{\mathbb R^d} \frac {|\mathfrak M f(x)|}{\log (e+ |\mathfrak Mf(x)|)+  \log(e+|x|)}dx <\infty.$$

\end{abstract}

\newtheorem{theorem}{Theorem}
\newtheorem{lemma}{Lemma}[section]
\newtheorem{proposition}{Proposition}[section]
\newtheorem{remark}{Remark}[section]
\newtheorem{corollary}{Corollary}[section]
\newtheorem{definition}{Definition}[section]
\newtheorem{example}{Example}[section]
\numberwithin{equation}{section}
\newtheorem{Theorem}{Theorem}
\newtheorem{Lemma}{Lemma}[section]
\newtheorem{Proposition}{Proposition}[section]
\newtheorem{Remark}{Remark}[section]
\newtheorem{Corollary}{Corollary}[section]
\newtheorem{Definition}{Definition}[section]
\newtheorem{Example}{Example}[section]
\newtheorem{Question}{Question}
\newtheorem*{theorema}{Theorem A}
\newtheorem*{theoremb}{Theorem B}
\newtheorem*{theoremc}{Theorem C}

\section{Introduction}

Products of functions in $H^1$ and $BMO$ have been firstly considered  by Bonami, Iwaniec, Jones and Zinsmeister in \cite{BIJZ}. Such products make sense as distributions, and can be written as the sum of an integrable function and a function in a weighted Hardy-Orlicz space. To be more precise, for $f\in H^1(\bR^d)$ and $g\in BMO(\bR^d)$, we define the product (in the distribution sense) $f \times g$ as the distribution whose action on the Schwartz function $\varphi\in\mathcal S(\bR^{d})$ is given by
\begin{equation}
\left\langle f \times g,\varphi\right\rangle:=\left\langle \varphi
g,f\right\rangle,
\end{equation}
where the second bracket stands for the duality bracket between $H^1(\bR^{d})$ and its dual $BMO(\bR^{d})$. It is then proven in \cite{BIJZ} that
\begin{equation}\label{nocancel}
f \times g\in L^{1}(\bR^{d})+ H^{\Xi}_\sigma(\bR^{d}).
\end{equation}
Here $ H^\Xi_\sigma(\bR^d)$ is the weighted Hardy-Orlicz space related to the Orlicz function
\begin{equation}\label{orl-log}
\Xi(t):=\frac t{\log (e+t)}
\end{equation}
and with weight $\sigma(x):=\frac{1}{\log (e+|x|)}$.

Let $L= -\Delta+ V$ be a Schr\"odinger operator on $\mathbb R^d$, $d\geq 3$, where $V$ is a nonnegative potential, $V\ne 0$, and belongs to the reverse H\"older class $RH_{d/2}$. In \cite{DGMTZ} and \cite{DZ}, Dziuba\'nski et al. introduced two kinds of function spaces associated with $L$. One is the Hardy space $H^1_L(\mathbb R^d)$, the other is the space $BMO_L(\mathbb R^d)$. They established in \cite{DGMTZ} that  the dual space of $H^1_L(\mathbb R^d)$ is just $BMO_L(\mathbb R^d)$. Unfortunately, as for the classical spaces $H^1(\mathbb R^d)$ and $BMO(\mathbb R^d)$, the pointwise products $fg$ of  functions $f\in H^1_L(\mathbb R^d)$  and functions $g\in BMO_L(\mathbb R^d)$ maybe not integrable. However, similarly to the classical setting, Li and Peng  showed in \cite{LP1} that  such products can be defined in the sense of distributions which action on the Schwartz function $\varphi\in\mathcal S(\bR^{d})$ is
\begin{equation}
\langle f\times g,\varphi\rangle:= \langle \varphi g, f\rangle,
\end{equation}
where the second bracket stands for the duality bracket between $H^1_L(\mathbb R^{d})$ and its dual $BMO_L(\mathbb R^{d})$. Moreover, they proved that $f\times g$ can be written as the sum of two distributions, one in $L^1(\mathbb R^d)$, the other in $H^\Xi_{L,\sigma}(\mathbb R^d)$ the weighted Hardy-Orlicz space associated with $L$ related to the Orlicz function $\Xi(t)\equiv  \frac{t}{\log(e+t)}$ and the weight $\sigma(x)\equiv \frac{1}{\log(e+|x|)}$, see Definition \ref{the definition of weighted Hardy-Orlicz spaces}. 

More precisely, in \cite{LP1}, the authors proved the following.

\begin{Theorem}\label{the main theorem of Li and Peng}
For each $f\in H^1_L(\mathbb R^d)$, there are two bounded linear operators $L_f: BMO_L(\mathbb R^d)\to L^1(\mathbb R^d)$ and $H_f: BMO_L(\mathbb R^d)\to H^{\Xi}_{L,\sigma}(\mathbb R^{d})$ such that for every $g\in BMO_L(\mathbb R^d)$, we have
\begin{equation}\label{linear decomposition for the product space}
f \times g = L_f(g)+ H_f(g)
\end{equation}
and the uniform bound
\begin{equation}\label{the uniform bound for linear decomposition}
\|L_f(g)\|_{L^1}+ \|H_f(g)\|_{H^{\Xi}_{L,\sigma}}\leq C \|f\|_{H^1_L}\|g\|_{BMO_L^+},
\end{equation}
where $\|g\|_{BMO_L^+}= \|g\|_{BMO_L}+ |g_{\mathbb B}|$, $g_{\mathbb B}$ denotes the mean value of $g$ over the unit ball $\mathbb B$.
\end{Theorem}

Our main theorem is as follows.

\begin{Theorem}\label{the main theorem}
There are two bounded bilinear operators $S_L: H^1_L(\mathbb R^d)\times BMO_L(\mathbb R^d)\to L^1(\mathbb R^d)$ and $T_L:  H^1_L(\mathbb R^d)\times BMO_L(\mathbb R^d)\to H^{\log}(\mathbb R^d)$ such that for every $(f,g)\in H^1_L(\mathbb R^d)\times BMO_L(\mathbb R^d)$, we have
\begin{equation}\label{bilinear decomposition for the product space}
f\times g = S_L(f,g) + T_L(f,g)
\end{equation}
and the uniform bound
\begin{equation}\label{the uniform bound for bilinear decomposition}
\|S_L(f,g)\|_{L^1}+ \|T_L(f,g)\|_{H^{\log}}\leq C \|f\|_{H^1_L}\|g\|_{BMO_L}.
\end{equation}
\end{Theorem}

Here $H^{\log}(\mathbb R^d)$ is a new kind of Hardy-Orlicz space consisting of all distributions $f$ such that  $\int_{\mathbb R^d} \frac{\mathfrak Mf(x)}{\log(e+ \mathfrak Mf(x)) + \log(e+ |x|)}dx<\infty$
with the  norm 
$$\|f\|_{H^{\log}}= \inf\left\{\lambda>0: \int_{\mathbb R^d} \frac{\frac{\mathfrak Mf(x)}{\lambda}}{\log\Big(e+ \frac{\mathfrak Mf(x)}{\lambda}\Big) + \log(e+ |x|)}dx\leq 1\right\}.$$
Recall that the grand maximal operator $\mathfrak M$ is defined by
\begin{equation}
\mathfrak Mf(x)= \sup_{\phi\in \mathcal A}\sup_{|y-x|<t}|f*\phi_t(y)|,
\end{equation}
where $\mathcal A= \{\phi\in \mathcal S(\mathbb R^d): |\phi(x)|+ |\nabla \phi(x)|\leq (1+|x|^2)^{-(d+1)}\}$ and $\phi_t(\cdot):= t^{-d}\phi(t^{-1}\cdot)$. 

Note that $H^{\log}(\mathbb R^d)\subset H^{\Xi}_{L,\sigma}(\mathbb R^{d})$ with continuous embedding, see Section \ref{the inclusion}. Compared with the main result of \cite{LP1} (Theorem \ref{the main theorem of Li and Peng}), our main result makes an essential improvement in two directions. The first one consists in proving that the space $H^{\Xi}_{L,\sigma}(\mathbb R^{d})$ can be replaced by a smaller space $H^{\log}(\mathbb R^d)$. Secondly, we give the bilinear decomposition (\ref{bilinear decomposition for the product space}) for the product space $H^1_L(\mathbb R^d)\times BMO_L(\mathbb R^d)$ instead of the linear decomposition (\ref{linear decomposition for the product space}) depending on $f\in H^1_L(\mathbb R^d)$. Moreover, we just need the $BMO_L$-norm (see (\ref{the uniform bound for bilinear decomposition})) instead of the $BMO_L^+$-norm as in (\ref{the uniform bound for linear decomposition}).

In applications to nonlinear PDEs, the distribution $f\times g\in \mathcal S'(\mathbb R^d)$ is used to justify weak continuity properties of the pointwise product $fg$. It is therefore important to recover $fg$ from the action of the distribution $f\times g$ on the test functions. An idea that naturally comes to mind is to look at the mollified distributions
\begin{equation}\label{bilinear decomposition theorem and distributions}
(f\times g)_\epsilon = (f\times g)*\phi_\epsilon,
\end{equation}
and let $\epsilon\to 0$. Here $\phi\in \mathcal S(\mathbb R^d)$ with $\int_{\mathbb R^d} \phi(x)dx=1$.

In the classical setting of $f\in H^1(\mathbb R^d)$ and $g\in BMO(\mathbb R^d)$, Bonami et al. proved in \cite{BIJZ} that the limit (\ref{bilinear decomposition theorem and distributions}) exists and equals $fg$ almost everywhere. An analogous result is also true for the Schr\"odinger setting. Namely, the following is true.

\begin{Theorem}\label{molified distributions}
Let $f\in H^1_L(\mathbb R^d)$ and $g\in BMO_L(\mathbb R^d)$. Then, for almost every $x\in\mathbb R^d$,
$$\lim_{\epsilon\to 0} (f\times g)_\epsilon(x) = f(x) g(x).$$
\end{Theorem}

Throughout the whole paper, $C$ denotes a positive geometric constant which is independent of the main parameters, but may change from line to line. 

The paper is organized as follows. In Section 2, we present some notations and preliminaries about Hardy type spaces associated with $L$. Section 3 is devoted to prove that $H^{\log}(\mathbb R^d)\subset H^{\Xi}_{L,\sigma}(\mathbb R^{d})$ with continuous embedding. Finally, the proofs of Theorem \ref{the main theorem} and Theorem \ref{molified distributions} are given in Section 4.

{\bf Acknowledgements.} The author would like to thank  Aline Bonami and  Sandrine Grellier for many helpful suggestions and discussions.

\section{Some preliminaries and notations}\label{Some preliminaries and notations}

In this paper, we consider the Schr\"odinger differential operator
$$L= -\Delta+ V$$
 on $\mathbb R^d$, $d\geq 3$, where $V$ is a nonnegative potential, $V\ne 0$. As in the works of Dziuba\'nski et al \cite{DGMTZ, DZ}, we always assume that $V$ belongs to the reverse H\"older class $RH_{d/2}$. Recall that a nonnegative locally integrable function $V$ is said to belong to a reverse H\"older class $RH_q$, $1<q<\infty$, if  there exists $C>0$ such that
$$\Big(\frac{1}{|B|}\int_B (V(x))^q dx\Big)^{1/q}\leq \frac{C}{|B|}\int_B V(x) dx$$
holds for every balls $B$ in $\mathbb R^d$.

Let $\{T_t\}_{t>0}$ be the semigroup generated by $L$ and $T_t(x,y)$ be their kernels. Namely,
$$T_t f(x)=e^{-t L}f(x)=\int_{\mathbb R^d} T_t(x,y)f(y)dy,\quad f\in L^2(\mathbb R^d),\quad t>0.$$

We say that a function $f\in L^2(\mathbb R^d)$ belongs to the space $\mathbb H^1_L(\mathbb R^d)$ if 
$$\|f\|_{\mathbb H^1_L}:= \|\mathcal M_L f\|_{L^1}<\infty,$$
where $\mathcal M_L f(x):= \sup_{t>0}|T_t f(x)|$ for all $x\in \mathbb R^d$. The space $H^1_L(\mathbb R^d)$ is then defined as the completion of $\mathbb H^1_L(\mathbb R^d)$ with respect to this norm.

In \cite{DGMTZ} it was shown that the dual of $H^1_{L}(\mathbb R^d)$ can be identified with the space $BMO_L(\mathbb R^d)$ which consists of all functions $f\in BMO(\mathbb R^d)$ with
$$\|f\|_{BMO_L} := \|f\|_{BMO}+\sup_{\rho(x)\leq r}\frac{1}{|B(x,r)|}\int_{B(x,r)}|f(y)|dy<\infty,$$
where $\rho$ is  the auxiliary function defined as in \cite{Sh}, that is,
\begin{equation}
\rho(x)= \sup\Big\{r>0: \frac{1}{r^{d-2}}\int_{B(x,r)} V(y)dy\leq 1\Big\},
\end{equation}
$x\in \mathbb R^d$. Clearly, $0<\rho(x)<\infty$ for all $x\in \mathbb R^d$, and thus $\mathbb R^d=\bigcup_{n\in\mathbb Z}\mathcal B_n$, where the sets $\mathcal B_n$ are defined by
\begin{equation}
\mathcal B_n= \{x\in \mathbb R^d: 2^{-(n+1)/2}< \rho(x)\leq 2^{-n/2}\}.
\end{equation}

The following proposition is due to Shen \cite{Sh}.

\begin{Proposition}
 [see \cite{Sh}, Lemma 1.4] \label{Shen, Lemma 1.4}
There exist $C_0>1$ and $k_0\geq 1$ such that for all $x,y\in\mathbb R^d$,
$$C_0^{-1}\rho(x) \Big(1+ \frac{|x-y|}{\rho(x)}\Big)^{-k_0}\leq \rho(y)\leq C_0 \rho(x) \Big(1+ \frac{|x-y|}{\rho(x)}\Big)^{\frac{k_0}{k_0+1}}.$$
\end{Proposition}

Here and in what follows, we denote by $\mathcal C_L$ the $L$-constant 
\begin{equation}\label{technique constant}
\mathcal C_L= 8. 9^{k_0}C_0
\end{equation}
where $k_0$ and $C_0$ are defined as in Proposition \ref{Shen, Lemma 1.4}.

\begin{Definition}
Given $1<q\leq \infty$. A function $a$ is called a $(H^1_L,q)$-atom related to the ball $ B(x_0,r)$ if $r\leq  \mathcal C_L \rho(x_0) $ and

i) supp $a\subset B(x_0,r)$,

ii) $\|a\|_{L^q}\leq |B(x_0,r)|^{1/q-1}$,

iii) if $r\leq \frac{1}{\mathcal C_L}\rho(x_0)$ then $\int_{\mathbb R^d}a(x)dx=0$.
\end{Definition}

The following atomic characterization of $H^1_L(\mathbb R^d)$ is due to Dziuba\'nski and Zienkiewicz \cite{DZ}.

\begin{theorema}
Let $1<q\leq \infty$. A function $f$ is in $H^1_L(\mathbb R^d)$ if and only if it can be written as $f=\sum_j \lambda_j a_j$, where $a_j$ are $(H^1_L,q)$-atoms and $\sum_j |\lambda_j|<\infty$. Moreover, there exists $C>1$ such that for every $f\in H^1_L(\mathbb R^d)$, we have
$$C^{-1}\|f\|_{H^1_L}\leq \inf\left\{\sum_j |\lambda_j|: f=\sum_j \lambda_j a_j\right\}\leq C\|f\|_{H^1_L}.$$
\end{theorema}

Let $1\leq q<\infty$. A nonnegative locally integrable function $w$ belongs to the {\sl Muckenhoupt class} $A_q$, say $w\in A_q$, if there exists a positive constant $C$ so that
\begin{equation}\label{Muckenhoupt class}
\frac{1}{|B|}\int_B w(x)dx\Big(\frac{1}{|B|}\int_B (w(x))^{-1/(q-1)}dx\Big)^{q-1}\leq C, \quad\mbox{if}\; 1<q<\infty,
\end{equation}
and 
\begin{equation}
\frac{1}{|B|}\int_B w(x)dx\leq C \mathop{\mbox{ess-inf}}\limits_{x\in B}w(x),\quad{\rm if}\; q=1,
\end{equation}
for all balls $B$ in $\mathbb R^d$. We say that $w\in A_\infty$ if $w\in A_q$ for some $q\in [1,\infty)$.

\begin{Remark}\label{the main weight}
The weight $\sigma (x)\equiv \frac{1}{\log(e+|x|)}$ belongs to the class $A_1$.
\end{Remark}

It is well known that $w\in A_p$, $1\leq p<\infty$, implies $w\in A_q$ for all $q >p$. For a measurable set $E$, we note  $w(E) =\int_E w(x) dx$ its weighted measure.

\begin{Definition}
Let $0<p\leq 1$. A function $\Phi$ is called a  growth function of order $p$ if it satisfies the following properties:

i) The function $\Phi$ is a Orlicz function, that is, $\Phi$ is a nondecreasing function with $\Phi(t)>0, t>0$, $\Phi(0)=0$  and $\lim_{t\to \infty}\Phi(t)=\infty$.

ii) The function $\Phi$ is of lower type $p$, that is, there exists a constant $C>0$ such that for every $s\in (0,1]$ and $t>0$,
$$\Phi(st)\leq C s^p \Phi(t).$$

iii) The function $\Phi$ is of upper type $1$, that is, there exists a constant $C>0$ such that for every $s\in [1,\infty)$ and $t>0$,
$$\Phi(st)\leq C s \Phi(t).$$
\end{Definition}

We will also say that $\Phi$ is a {\sl growth function} whenever it is a {\sl growth function} of some order $p<1$.

\begin{Remark}\label{sublinear property}
i) Let $\Phi$ be a growth function. Then, there exists a constant $C>0$ such that 
$$\Phi\Big(\sum_{j=1}^\infty t_j\Big)\leq C \sum_{j=1}^\infty \Phi(t_j)$$
for every sequence $\{t_j\}_{j\geq 1}$ of nonnegative real numbers. See Lemma 4.1 of \cite{Ky1}.

ii) The function $\Xi (t)\equiv \frac{t}{\log(e+t)}$ is a growth function of order $p$ for any $p\in (0,1)$.
\end{Remark}

Now, let us define weighted Hardy-Orlicz spaces associated with $L$.

\begin{Definition}\label{the definition of weighted Hardy-Orlicz spaces}
Given $w\in A_\infty$ and $\Phi$  a growth function. We say that a function $f\in L^2(\mathbb R^d)$ belongs to  $\mathbb H^\Phi_{L,w}(\mathbb R^d)$ if $\int_{\mathbb R^d} \Phi(\mathcal M_Lf(x))w(x)dx<\infty$. The space $H^\Phi_{L,w}(\mathbb R^d)$ is  defined as the completion of  $\mathbb H^\Phi_{L,w}(\mathbb R^d)$ with respect to the norm
$$\|f\|_{ H^\Phi_{L,w}}:= \inf\left\{\lambda>0: \int_{\mathbb R^d} \Phi\Big(\frac{\mathcal M_Lf(x)}{\lambda}\Big)w(x)dx\leq 1\right\}.$$
\end{Definition}

Remark that when $w(x)\equiv 1$ and $\Phi(t)\equiv t$, the space $H^\Phi_{L,w}(\mathbb R^d)$ is just $H^1_L(\mathbb R^d)$. We refer the reader to the recent work of D. Yang and S. Yang \cite{YY} for a complete study of the theory of weighted Hardy-Orlicz spaces associated with operators.

\section{The inclusion $H^{\log}(\mathbb R^d)\subset H^{\Xi}_{L,\sigma}(\mathbb R^{d})$}\label{the inclusion}

The purpose of this section is to establish the following embedding.

\begin{Proposition}\label{the embedding}
$H^{\log}(\mathbb R^d)\subset H^{\Xi}_{L,\sigma}(\mathbb R^{d})$ and the inclusion is continuous.
\end{Proposition}

Recall (see \cite{Ky1}) that the weighted Hardy-Orlicz space $H^{\Xi}_{\sigma}(\mathbb R^d)$ is defined as the space of all distributions $f$ such that  $\int_{\mathbb R^d} \frac{\mathfrak Mf(x)}{\log(e+ \mathfrak Mf(x))}\frac{1}{\log(e+ |x|)}dx<\infty$
with the  norm 
$$\|f\|_{H^{\Xi}_{\sigma}}= \inf\left\{\lambda>0: \int_{\mathbb R^d} \frac{\frac{\mathfrak Mf(x)}{\lambda}}{\log\Big(e+ \frac{\mathfrak Mf(x)}{\lambda}\Big)}\frac{1}{\log(e+ |x|)}dx\leq 1\right\}.$$

Clearly, $H^{\log}(\mathbb R^d)\subset H^{\Xi}_{\sigma}(\mathbb R^{d})$ and the inclusion is continuous. Consequently, {\sl the proof of Proposition \ref{the embedding}} can be reduced to showing that for every $f\in H^{\Xi}_{\sigma}(\mathbb R^{d})$,
\begin{equation}\label{the main estimate}
\|f\|_{H^{\Xi}_{L,\sigma}}\leq C \|f\|_{H^{\Xi}_{\sigma}}.
\end{equation}

Let $1<q\leq\infty$. Recall (see \cite{Ky1}) that a function $a$ is called a $(H^\Xi_\sigma,q)$-atom related to the ball $B$ if

i) supp $a\subset B$,

ii) $\|a\|_{L^q_\sigma}\leq \sigma(B)^{1/q}\Xi^{-1}(\sigma(B)^{-1})$, where $\Xi^{-1}$ is the inverse function of $\Xi$,

iii) $\int_{\mathbb R^d}a(x)dx=0$.

In order to prove Proposition \ref{the embedding}, we need the following lemma.

\begin{Lemma}\label{the key lemma for the inclusion}
Let $1<q<\infty$. Then, 
\begin{equation}\label{key lemma}
\int_{\mathbb R^d} \Xi(\mathcal M_L f(x))\sigma(x)dx\leq C \sigma(B) \Xi(\sigma(B)^{-1/q}\|f\|_{L^q_{\sigma}})
\end{equation}
for every $f$ multiples of $(H^\Xi_\sigma,q)$-atom related to the ball $B=B(x_0,r)$, 
\end{Lemma}

To prove Lemma \ref{the key lemma for the inclusion}, let us recall the following.

\begin{Lemma}[see \cite{LP1}, Lemma 2]\label{LP1, Lemma 2}
Let $V\in RH_{d/2}$. Then, there exists $\delta>0$ depends only on $L$, such that for every $|y-z|<|x-y|/2$ and $t>0$, we have
$$|T_t(x,y)- T_t(x,z)|\leq C\Big(\frac{|y-z|}{\sqrt t}\Big)^{\delta}t^{-\frac{d}{2}}e^{-\frac{|x-y|^2}{t}}\leq C\frac{|y-z|^{\delta}}{|x-y|^{d+\delta}}.$$
\end{Lemma}

\begin{proof}[Proof of Lemma \ref{the key lemma for the inclusion}]
First, note that $\sigma\in A_1$ and $\Xi$ is a {\sl growth function} of order $p$ for any $p\in (0,1)$, see Remark \ref{the main weight} and Remark \ref{sublinear property}. Denote by $\mathcal M$ the classical Hardy-Littlewood maximal operator. Then, the estimate $\mathcal M_Lf\leq C \mathcal Mf$, the $L^q_\sigma$-boundedness of $\mathcal M$ and H\"older inequality give
\begin{eqnarray}\label{key lemma 1}
&&\int_{B(x_0,2r)} \Xi(\mathcal M_Lf(x))\sigma(x)dx \nonumber\\
&\leq& C \int_{B(x_0,2r)} \Xi(\mathcal Mf(x)+ \sigma(B)^{-1/q}\|f\|_{L^q_\sigma})\sigma(x)dx  \nonumber\\
 &\leq& C \int_{B(x_0,2r)} \Big(\frac{\mathcal Mf(x)+ \sigma(B)^{-1/q}\|f\|_{L^q_\sigma}}{\sigma(B)^{-1/q}\|f\|_{L^q_\sigma}}\Big)\Xi(\sigma(B)^{-1/q}\|f\|_{L^q_\sigma})\sigma(x)dx\nonumber\\
&\leq& C \sigma(B) \Xi(\sigma(B)^{-1/q}\|f\|_{L^q_\sigma}),
\end{eqnarray}
where we used the facts that $t\mapsto \frac{\Xi(t)}{t}$ is nonincreasing and  $\sigma(B(x_0, 2r))\leq C\sigma(B)$.

Let $x\notin B(x_0,2r)$ and $t>0$. By Lemma \ref{LP1, Lemma 2} and (\ref{Muckenhoupt class}), 
\begin{eqnarray*}
|T_tf(x)|=\Big|\int_{\mathbb R^d} T_t(x,y) f(y)dy\Big|&=& \Big|\int_{B}(T_t(x,y)- T_t(x,x_0))f(y)dy\Big|\\
&\leq& C \int_B \frac{|y-x_0|^{\delta}}{|x-x_0|^{d+\delta}}|f(y)|dy\\
&\leq& C \sigma(B)^{-1/q}\|f\|_{L^q_\sigma}\frac{r^{d+\delta}}{|x-x_0|^{d+\delta}}
\end{eqnarray*}
Therefore, as $\Xi$ is of lower type $\frac{2d+\delta}{2(d+\delta)}<1$, 
\begin{eqnarray}\label{key lemma 2}
&&\int_{(B(x_0,2r))^c} \Xi(\mathcal M_Lf(x))\sigma (x)dx\nonumber\\
 &\leq& C \Xi(\sigma(B)^{-1/q}\|f\|_{L^q_\sigma})\int_{(B(x_0,2r))^c}\Big(\frac{r^{d+\delta}}{|x-x_0|^{d+\delta}}\Big)^{\frac{2d+\delta}{2(d+\delta)}}\sigma (x)dx\nonumber\\
&\leq& C \sigma(B) \Xi(\sigma(B)^{-1/q}\|f\|_{L^q_\sigma}),
\end{eqnarray}
where we used (see \cite{GR}, page 412)
$$\int_{(B(x_0, 2r))^c}\frac{r^{d+\delta/2}}{|x-x_0|^{d+\delta/2}}\sigma(x)dx\leq C \sigma(B(x_0,2r))\leq C \sigma(B).$$ 

Then, (\ref{key lemma}) follows from (\ref{key lemma 1}) and (\ref{key lemma 2}). This completes the proof.

\end{proof}

\begin{proof}[Proof of Proposition \ref{the embedding}]
As mentioned above, it is sufficient to show that
$$\|f\|_{H^{\Xi}_{L,\sigma}}\leq C \|f\|_{H^\Xi_\sigma}$$
for every $f\in H^\Xi_\sigma(\mathbb R^d)$. By Theorem 3.1 of \cite{Ky1}, there are multiples of $(H^\Xi_\sigma,2)$-atoms $b_j$, $j=1,2,...$, related to balls $B_j$ such that $f= \sum_{j=1}^\infty b_j$ and 
\begin{equation}\label{key proposition 1}
\Lambda_2(\{b_j\})\leq C \|f\|_{H^\Xi_\sigma},
\end{equation}
where
$$\Lambda_2(\{b_j\}):= \inf\left\{\lambda>0: \sum_{j=1}^\infty \sigma(B_j)\Xi\Big(\frac{\sigma(B_j)^{-1/2}\|b_j\|_{L^2_\sigma}}{\lambda}\Big)\leq 1 \right\}.$$

On the other hand, the estimate $\mathcal M_Lf\leq \sum_{j=1}^\infty \mathcal M_L(b_j)$, Remark \ref{sublinear property} and  Lemma \ref{the key lemma for the inclusion} give
\begin{eqnarray}
\int_{\mathbb R^d}  \Xi\Big(\frac{\mathcal M_Lf(x)}{\Lambda_2(\{b_j\})}\Big)\sigma(x)dx &\leq& C\sum_{j=1}^\infty \int_{\mathbb R^d}  \Xi\Big(\frac{\mathcal M_L(b_j)(x)}{\Lambda_2(\{b_j\})}\Big)\sigma(x)dx\nonumber\\
&\leq& C \sum_{j=1}^\infty \sigma(B_j)\Xi\Big(\frac{\sigma(B_j)^{-1/2}\|b_j\|_{L^q_\sigma}}{\Lambda_2(\{b_j\})}\Big)\nonumber\\
&\leq& C,\nonumber
\end{eqnarray}
which implies that $\|f\|_{H^\Xi_{L,\sigma}}\leq C \Lambda_2(\{b_j\})$. Therefore, (\ref{key proposition 1}) yields
$$\|f\|_{H^\Xi_{L,\sigma}}\leq C \|f\|_{H^\Xi_\sigma},$$
which completes the proof of Proposition \ref{the embedding}.

\end{proof}

\section{Proof of Theorem \ref{the main theorem} and Theorem \ref{molified distributions}}

Let $P(x)= (4\pi)^{-d/2} e^{-|x|^2/4}$ be the Gauss function. For $n\in\mathbb Z$, following \cite{DZ}, the space $h^1_n(\mathbb R^d)$  denotes the space of all integrable functions $f$ such that
$$\mathcal  M_nf(x) =\sup_{0<t<2^{-n}} |P_{\sqrt t}*f(x)|=\sup_{0<t<2^{-n}}\Big|\int_{\mathbb R^d} p_t(x,y) f(y)dy\Big| \in L^1(\mathbb R^d),$$
where the kernel $p_t$ is given by  $p_t(x,y)= (4\pi t)^{-d/2} e^{-\frac{|x-y|^2}{4t}}$. We equipped this space with the norm $\|f\|_{h^1_n}:= \|\mathcal M_n f\|_{L^1}$.

For convenience of the reader, we list here some lemmas used in our proofs.

\begin{Lemma}[see \cite{DZ}, Lemma 2.3] \label{DZ, Lemma 2.3}
There exists a constant $C>0$ and a collection of balls $B_{n,k}= B(x_{n,k}, 2^{-n/2})$, $n\in\mathbb Z, k=1,2,...$, such that $x_{n,k}\in \mathcal B_n$, $\mathcal B_n\subset \bigcup_k B_{n,k}$, and 
$$card \, \{(n',k'): B(x_{n,k}, R 2^{-n/2})\cap B(x_{n',k'}, R 2^{-n/2})\ne \emptyset\}\leq R^C$$
 for all $n,k$ and $R\geq 2$.
\end{Lemma}

\begin{Lemma}[see \cite{DZ}, Lemma 2.5]\label{DZ, Lemma 2.5}
There are nonnegative $C^\infty$-functions $\psi_{n,k}$, $n\in\mathbb Z, k=1,2,...$, supported in the balls $B(x_{n,k}, 2^{1-n/2})$ such that
$$\sum_{n,k}\psi_{n,k}=1\quad\mbox{and}\quad \|\nabla \psi_{n,k}\|_{L^\infty}\leq C 2^{n/2}.$$
\end{Lemma}

\begin{Lemma}[see (4.7) in \cite{DZ}]\label{DZ}\label{DZ, 4.7}
For every $f\in H^1_L(\mathbb R^d)$, we have
$$\sum_{n,k} \|\psi_{n,k}f\|_{h^1_n}\leq C \|f\|_{H^1_L}.$$
\end{Lemma}

In this section, we fix a non-negative function $\varphi\in \mathcal S(\mathbb R^d)$ with supp $\varphi\subset B(0,1)$ and $\int_{\mathbb R^d}\varphi(x)dx=1$. Then, we define the linear operator $\mathfrak H$  by
$$\mathfrak H(f)= \sum_{n,k}\Big(\psi_{n,k}f- \varphi_{2^{-n/2}}*(\psi_{n,k}f)\Big).$$

In order to prove Theorem \ref{the main theorem}, we need two key lemmas.

\begin{Lemma}\label{key lemma H}
The operator $\mathfrak H$ maps continuously $H^1_L(\mathbb R^d)$ into $H^1(\mathbb R^d)$.
\end{Lemma}

The proof of Lemma \ref{key lemma H} can be found in \cite{Ky2} (see Lemma 5.1 of \cite{Ky2}).

\begin{Lemma}\label{the key lemma}
There exists a constant $C= C(\varphi,d)>0$ such that for all $(n,k)\in\mathbb Z\times \mathbb Z^+$, $g\in BMO_L(\mathbb R^d)$  and $f\in h^1_n(\mathbb R^d)$ with supp $f\subset B(x_{n,k}, 2^{1-n/2})$, we have
$$\Big\|(\varphi_{2^{-n/2}}*f)g\Big\|_{H^1_L}\leq C \|f\|_{h^1_n} \|g\|_{BMO_L}.$$
\end{Lemma}

To prove Lemma \ref{the key lemma}, we need the following.

\begin{Lemma}[see \cite{Ky2},  Lemma 6.5]\label{Ky2, Lemma 6.5}
Let $1<q\leq \infty$,  $n\in\mathbb Z$ and $x\in \mathcal B_n$. Suppose that $f\in h^1_n(\mathbb R^d)$ with supp $f\subset B(x, 2^{1-n/2})$. Then, there are $(H^1_L,q)$-atoms $a_j$ related to the balls $B(x_j,r_j)$ such that $ B(x_j,r_j)\subset B(x, 2^{2-n/2})$ and
$$f= \sum_j \lambda_j a_j, \quad \sum_j |\lambda_j|\leq C \|f\|_{h^1_n}$$
with a positive constant $C$ independent of $n$ and $f$.
\end{Lemma}

Here and in what follows, for any $B$ a ball in $\mathbb R^d$ and $f$ a locally integrable function, we denote by $f_B$ the average of $f$ on $B$.

\begin{proof}[Proof of Lemma \ref{the key lemma}]
As $x_{n,k}\in \mathcal B_n$, it follows from Lemma \ref{Ky2, Lemma 6.5} that there are $(H^1_L,2)$-atoms $a_j^{n,k}$ related to the balls $B(x_j^{n,k},r_j^{n,k})\subset B(x_{n,k}, 2^{2-n/2})$ such that
\begin{equation}\label{the key lemma 1}
f=\sum_j \lambda_j^{n,k} a_j^{n,k}\quad\mbox{and}\quad \sum_j |\lambda_j^{n,k}|\leq C \|f\|_{h^1_n},
\end{equation}
where the positive constant $C$ is independent of $f,n,k$.

Now, let us establish that $\varphi_{2^{-n/2}}*a_j^{n,k}$ is $C$ times a $(H^1_L,2)$-atom related to the ball $B(x_{n,k}, 5. 2^{-n/2})$. Indeed, it is clear that $\frac{1}{\mathcal C_L}\rho(x_{n,k})< 5. 2^{-n/2}< \mathcal C_L \rho(x_{n,k})$ since $x_{n,k}\in \mathcal B_n$; and  supp $\varphi_{2^{-n/2}}*a_j^{n,k}\subset B(x_{n,k}, 5. 2^{-n/2})$ since supp $\varphi\subset B(0,1)$ and supp $a_j^{n,k}\subset B(x_{n,k}, 2^{2-n/2})$. In addition, 
$$\|\varphi_{2^{-n/2}}*a_j^{n,k}\|_{L^2}\leq \|\varphi_{2^{-n/2}}\|_{L^2}\|a_j^{n,k}\|_{L^1}\leq (2^{-n/2})^{-d/2}\|\varphi\|_{L^2}\leq C |B(x_{n,k},5.2^{-n/2})|^{-1/2}.$$
These prove that $\varphi_{2^{-n/2}}*a_j^{n,k}$ is $C$ times a $(H^1_L,2)$-atom related to $B(x_{n,k}, 5. 2^{-n/2})$.

By an analogous argument, it is easy to check that $(\varphi_{2^{-n/2}}*a_j^{n,k})(g - g_{B(x_{n,k}, 5. 2^{-n/2})})$ is $C\|g\|_{BMO}$ times a  $(H^1_L,3/2)$-atom related to $B(x_{n,k}, 5. 2^{-n/2})$.

 Therefore, (\ref{the key lemma 1}) yields
\begin{eqnarray*}
\Big\|(\varphi_{2^{-n/2}}*f)g\Big\|_{H^1_L}&\leq& C \sum_j |\lambda_j|\|(\varphi_{2^{-n/2}}*a_j) (g- g_{B(x_{n,k}, 5. 2^{-n/2})})\|_{H^1_L}\\
&& + C\sum_j |\lambda_j|\|\varphi_{2^{-n/2}}*a_j\|_{H^1_L}  |g_{B(x_{n,k}, 5. 2^{-n/2})}|\\
&\leq& C\|f\|_{h^1_n}\|g\|_{BMO_L},
\end{eqnarray*}
where we used $|g_{B(x_{n,k}, 5. 2^{-n/2})}|\leq \|g\|_{BMO_L}$ since $\rho(x_{n,k})\leq 5. 2^{-n/2}$. 

\end{proof}

Our main results are strongly related to the recent result of Bonami, Grellier and Ky \cite{BGK}. In \cite{BGK}, the authors proved the following.

\begin{Theorem}\label{theorem of Bonami, Grellier and Ky}
There exists two continuous bilinear operators on the product space $H^1(\mathbb R^d)\times BMO(\mathbb R^d)$, respectively   $S:H^1(\bR^d)\times BMO(\bR^d) \mapsto L^{1}(\bR^{d})$ and $T:H^1(\bR^d)\times BMO(\bR^d) \mapsto H^{\log}(\bR^{d})$
 such that
$$ f\times g=S(f,g)+T(f,g).$$
\end{Theorem}

Before giving the proof of the main theorems, we should point out that the bilinear operator $T$ in Theorem \ref{theorem of Bonami, Grellier and Ky}  satisfies
\begin{equation}\label{lost 1}
\|T(f,g)\|_{H^{\rm log}}\leq C \|f\|_{H^1}(\|g\|_{BMO}+ |g_{\mathbb Q}|)
\end{equation}
where $\mathbb Q:= [0,1)^d$ is the unit cube. To prove this, the authors in \cite{BGK} used the generalized H\"older inequality (see also \cite{BIJZ})
$$\|fg\|_{L^{\rm log}}\leq C \|f\|_{L^1}\|g\|_{\rm Exp}$$
and the fact that $\|g - g_{\mathbb Q}\|_{\rm Exp} \leq C \|g\|_{BMO}$. Here, $L^{\log}(\bR^{d})$ denotes  the space of all measurable functions $f$ such that $\int_{\mathbb R^d} \frac {|f(x)|}{\log (e+ |f(x)|) + \log(e+|x|) }dx <\infty$ with the norm
$$\|f\|_{L^{\log}}= \inf\left\{\lambda>0: \int_{\mathbb R^d} \frac {\frac{|f(x)|}{\lambda}}{\log (e+ \frac{|f(x)|}{\lambda}) + \log(e+|x|)}dx \leq 1\right\}$$
and  ${\rm Exp}(\mathbb R^d)$ denotes the space of all measurable functions $f$ such that $\int_{\mathbb R^d}(e^{|f(x)|}-1)\frac{1}{(1+|x|)^{2d}}dx<\infty$ with the norm
$$\|f\|_{\rm Exp}= \inf\left\{\lambda>0: \int_{\mathbb R^d}\Big(e^{|f(x)|/\lambda}-1\Big)\frac{1}{(1+|x|)^{2d}}dx\leq 1\right\}.$$

In fact, Inequality (\ref{lost 1}) also holds when we replace the unit cube $\mathbb Q$ by $B(0,r)$ for every $ r > 0$  since  $\|g - g_{B(0,r)}\|_{\rm Exp} \leq C \|g\|_{BMO}$. More precisely, there exists a constant $C > 0$ such that
\begin{equation}\label{log-estimate}
\|fg\|_{L^{\rm log}}\leq C \|f\|_{L^1}(\|g\|_{BMO}+ |g_{B(0, \rho(0))}|)\leq C  \|f\|_{L^1}\|g\|_{BMO_L}
\end{equation}
for all $f\in L^1(\mathbb R^d)$ and $g\in BMO_L(\mathbb R^d)$. As a consequence, we obtain
\begin{equation}\label{lost 2}
\|T(f,g)\|_{H^{\rm log}}\leq C  \|f\|_{H^1}\|g\|_{BMO_L}
\end{equation}
for all $f \in H^1(\mathbb R^d)$ and $g\in BMO_L(\mathbb R^d)$.

Now, we are ready to give the proof of the main theorems.

\begin{proof}[Proof of Theorem \ref{the main theorem}]
We define two bilinear operators $S_L$ and $T_L$ by
$$S_L(f,g)= S(\mathfrak H(f), g) + \sum_{n,k} \Big(\varphi_{2^{-n/2}}*(\psi_{n,k}f)\Big)g$$
and
$$T_L(f,g)=  T(\mathfrak H(f), g)$$
for all $(f,g)\in H^1_L(\mathbb R^d)\times BMO_L(\mathbb R^d)$. Then, it follows from  Theorem \ref{theorem of Bonami, Grellier and Ky}, Lemma \ref{DZ}, Lemma \ref{key lemma H} and Lemma \ref{the key lemma} that
\begin{eqnarray*}
\|S_L(f,g)\|_{L^1}
&\leq&
\|S(\mathfrak H(f), g)\|_{L^1}+ C \sum_{n,k}\Big\|\Big(\varphi_{2^{-n/2}}*(\psi_{n,k}f)\Big)g\Big\|_{H^1_L}\\
&\leq& C \|g\|_{BMO}\|\mathfrak H(f)\|_{H^1}+ C \|g\|_{BMO_L}\sum_{n,k}\|\psi_{n,k}f\|_{h^1_n}\\
&\leq& C \|f\|_{H^1_L} \|g\|_{BMO_L},
\end{eqnarray*}
and as (\ref{lost 2}), 
\begin{eqnarray*}
\|T_L(f,g)\|_{H^{\rm log}}= \|T(\mathfrak H(f),g)\|_{H^{\rm log}} &\leq& C \|\mathfrak H(f)\|_{H^1} \|g\|_{BMO_L}\\
&\leq& C \|f\|_{H^1_L} \|g\|_{BMO_L}.
\end{eqnarray*}
Furthermore, in the sense of distributions, we have
\begin{eqnarray*}
&&S_L(f,g)+ T_L(f,g)\\
&=& \left(\sum_{n,k} \Big(\psi_{n,k}f- \varphi_{2^{-n/2}}*(\psi_{n,k}f)\Big)\right)\times g + \sum_{n,k} \Big(\varphi_{2^{-n/2}}*(\psi_{n,k}f)\Big)g\\
&=& \left(\sum_{n,k}\psi_{n,k}f\right)\times g = f\times g,
\end{eqnarray*}
which ends the proof of Theorem \ref{the main theorem}.

\end{proof}

\begin{proof}[Proof of Theorem \ref{molified distributions}]
By the proof of Theorem \ref{the main theorem}, the function $\sum_{n,k} (\varphi_{2^{-n/2}}*(\psi_{n,k}f))g$ belongs to $H^1_L(\mathbb R^d)\subset L^1(\mathbb R^d)$. This implies that $\Big(\sum_{n,k} (\varphi_{2^{-n/2}}*(\psi_{n,k}f))g\Big)*\phi_\epsilon$ tends to $\sum_{n,k} (\varphi_{2^{-n/2}}*(\psi_{n,k}f))g$ almost everywhere, as $\epsilon \to 0$. Therefore, applying Theorem 1.8 of \cite{BIJZ}, we get
\begin{eqnarray*}
\lim_{\epsilon\to 0}(f\times g)_\epsilon(x) &=& \lim_{\epsilon\to 0}(\mathfrak H(f)\times g)_\epsilon(x) + \lim_{\epsilon\to 0}\Big(\sum_{n,k} (\varphi_{2^{-n/2}}*(\psi_{n,k}f))g\Big)*\phi_\epsilon(x)\\
&=&  \mathfrak H(f)(x) g(x) + \left(\sum_{n,k} (\varphi_{2^{-n/2}}*(\psi_{n,k}f))(x)\right)g(x)\\
&=& f(x) g(x)
\end{eqnarray*}
for almost every $x\in \mathbb R^d$, which completes the proof of Theorem \ref{molified distributions}.

\end{proof}

\medskip
\noindent         Department of Mathematics, University of Quy Nhon, 170 An Duong Vuong, Quy Nhon, Binh Dinh, Viet Nam \\
Email: dangky@math.cnrs.fr


\begin{thebibliography}{MTW1}

\bibitem{BGK} A. Bonami, S. Grellier and L. D. Ky, Paraproducts and products of functions in $BMO(\mathbb R^n)$ and $H^1(\mathbb R^n)$ through wavelets,  J. Math. Pure Appl. {\bf 97}, 230--241 (2012).

\bibitem{BIJZ} A. Bonami, T. Iwaniec, P. Jones and M. Zinsmeister,  On the product of functions in $BMO$ and $H^1$, Ann. Inst. Fourier (Grenoble). {\bf 57}(5), 1405--1439 (2007).


\bibitem{DGMTZ} J. Dziuba\'nski, G. Garrig\'os, T. Mart\'inez, J. Torrea and J. Zienkiewicz, $BMO$ spaces related to Schr\"odinger operators with potentials satisfying a reverse H\"older inequality, Math. Z. {\bf 249}, 329--356 (2005).

\bibitem{DZ} J. Dziuba\'nski and J. Zienkiewicz, Hardy space $H^1$ associated to Schr\"odinger operator with potential satisfying reverse H\"older inequality, Rev. Mat. Iberoamericana. {\bf 15}, 279--296  (1999).

\bibitem{GR} J. Garc\'ia-Cuerva and J. L. Rubio de Francia, Weighted norm inequalities and related topics (Elsevier Science Ltd, Amsterdam-New Yord-Oxford, 1985).


\bibitem{Ky1} L. D. Ky, New Hardy spaces of Musielak-Orlicz type and boundedness of sublinear operators, Integral Equations and Operator Theory (to appear) or arXiv:1103.3757.

\bibitem{Ky2} L. D. Ky, Endpoint estimates for commutators of singular integrals related to Schr\"odinger operators, arXiv:1203.6335.

\bibitem{LP1} P. Li and L. Peng, The decomposition of product space $H^1_L\times BMO_L$. J. Math. Anal. Appl. {\bf 349}, 484--492  (2009).

\bibitem{Sh} Z. Shen, $L^p$ estimates for Schr\"odinger operators with certain potentials, Ann. Inst. Fourier. {\bf 45}(2), 513--546 (1995).

\bibitem{YY} D. Yang and S. Yang, Musielak-Orlicz-Hardy Spaces Associated with Operators and Their Applications,   J. Geom. Anal. (2012), 10.1007/s12220-012-9344-y.




\end{thebibliography}
\end{document}